\documentclass[11pt]{article}
\usepackage{amsfonts,amsmath,tikz}
\usepackage{epsfig}
\usepackage{latexsym}
\usepackage{color}
\usepackage{epsf,amssymb}
\usepackage{enumitem}
\usepackage{array}

\usepackage{amsthm}
\usepackage{yfonts}
\usepackage{hyperref}
\usepackage{ifthen}
\usepackage{booktabs}

\usepackage{cite}

\usepackage{expl3}

\usetikzlibrary{shapes,positioning}
\usetikzlibrary{calc}

\usepackage{float} 
\floatstyle{plaintop}

\newtheorem{theorem}{Theorem}
\newtheorem{lemma}[theorem]{Lemma}

\newtheorem{corollary}[theorem]{Corollary}
\newtheorem{proposition}[theorem]{Proposition}

\newcommand{\cind}{c_{\rm ind}}

\newcommand{\gammaenaL}{\gamma_1^{\mathrm L}}

\NewDocumentCommand{\Goldbergsnark}{m m m m m O{} O{} o}{
\tikzset{Style empty/.style={draw, circle, fill=white,  scale=0.4}}
\tikzset{Style full/.style={draw, circle, fill=black, scale=0.4}}

    \foreach \i in {1,...,#1} {
        \ifthenelse{\i=1}{            
            \coordinate (la{#5}\i) at ({(#3 + #2*3*(\i-1) - #2/2)},#2*3 + #4);
            \coordinate (lf{#5}\i) at ({(#3 + #2*3*(\i-1) - #2/2)},#2 + #4);
            \coordinate (lh{#5}\i) at ({(#3 + #2*3*\i -#2*2 - #2/2)},#2*0 + #4);

            \coordinate (lca{#5}\i) at ({(#3 + #2*3*(\i-1) - #2/2)},#2*3.75 + #4);
            \coordinate (lcf{#5}\i) at ({(#3 + #2*3*(\i-1) - #2/2)},#2-#2*10/4 + #4);
            \coordinate (lch{#5}\i) at ({(#3 + #2*3*\i -#2*2 - #2/2)},#2*0-#2*3/4 + #4);
        }{}
        \ifthenelse{\i=#1}{
            \coordinate (rb{#5}\i) at ({(#3 + #2*3*\i -#2 + #2/2)},#2*3 + #4);
            \coordinate (rg{#5}\i) at ({(#3 + #2*3*\i -#2 + #2/2)},#2 + #4);
            \coordinate (rh{#5}\i) at ({(#3 + #2*3*\i -#2*2 + #2/2)},#2*0 + #4);

            \coordinate (rcb{#5}\i) at ({(#3 + #2*3*\i -#2 + #2/2)},#2*3.75 + #4);
            \coordinate (rcg{#5}\i) at ({(#3 + #2*3*\i -#2 + #2/2)},#2-#2*10/4 + #4);
            \coordinate (rch{#5}\i) at ({(#3 + #2*3*\i -#2*2 + #2/2)},#2*0-#2*3/4 + #4);
        }{}
        
        \coordinate (a{#5}\i) at ({(#3 + #2*3*(\i-1))},#2*3 + #4);
        \coordinate (b{#5}\i) at ({(#3 + #2*3*\i -#2)},#2*3 + #4);
        \coordinate (c{#5}\i) at ({(#3 + #2*3*(\i-1))},#2*2 + #4);
        \coordinate (d{#5}\i) at ({(#3 + #2*3*\i -#2*2)},#2*2 + #4);
        \coordinate (e{#5}\i) at ({(#3 + #2*3*\i -#2)},#2*2 + #4);
        \coordinate (f{#5}\i) at ({(#3 + #2*3*(\i-1))},#2 + #4);
        \coordinate (g{#5}\i) at ({(#3 + #2*3*\i -#2)},#2 + #4);
        \coordinate (h{#5}\i) at ({(#3 + #2*3*\i -#2*2)},#2*0 + #4);
          
        \draw (a{#5}\i) -- (b{#5}\i);
        \draw (a{#5}\i) -- (e{#5}\i);
        \draw (b{#5}\i) -- (c{#5}\i);
        \draw (c{#5}\i) -- (d{#5}\i);
        \draw (c{#5}\i) -- (f{#5}\i);
        \draw (d{#5}\i) -- (e{#5}\i);
        \draw (d{#5}\i) -- (h{#5}\i);
        \draw (e{#5}\i) -- (g{#5}\i);
        \draw (f{#5}\i) -- (g{#5}\i);

        \pgfmathtruncatemacro{\pI}{\i-1}

        \ifthenelse{\i>1}
        {
            \draw (b{#5}\pI) -- (a{#5}\i);
            \draw (g{#5}\pI) -- (f{#5}\i);
            \draw (h{#5}\pI) -- (h{#5}\i);
        }{}

        \IfNoValueT{#8}
        {
            \ifthenelse{\i=1}
            {
                \draw (la{#5}\i) -- (a{#5}\i);
                \draw (lf{#5}\i) -- (f{#5}\i);
                \draw (lh{#5}\i) -- (h{#5}\i);
            }
    
            \ifthenelse{\i=#1}
            {
                \draw (rb{#5}\i) -- (b{#5}\i);
                \draw (rg{#5}\i) -- (g{#5}\i);
                \draw (rh{#5}\i) -- (h{#5}\i);
            }
        }
    }

    \IfNoValueF{#8}
    {
        \draw[rounded corners] (a{#5}1) -- (la{#5}1) -- (lca{#5}1) -- (rcb{#5}#1) -- (rb{#5}#1) -- (b{#5}#1);
        \draw[rounded corners] (h{#5}1) -- (lh{#5}1) -- (lch{#5}1) -- (rch{#5}#1) -- (rh{#5}#1) -- (h{#5}#1);
        \draw[rounded corners] (f{#5}1) -- (lf{#5}1) -- (lcf{#5}1) -- (rcg{#5}#1) -- (rg{#5}#1) -- (g{#5}#1);
     
    }
    
    \foreach \i in {1,...,#1} {
        \IfInSetTF{#6}{(\i-1)*8+1} 
        { \node[Style full] at (a{#5}\i) {}; }
        { \node[Style empty] at (a{#5}\i) {}; }
        
        \IfInSetTF{#6}{(\i-1)*8+2} 
        { \node[Style full] at (b{#5}\i) {}; }
        { \node[Style empty] at (b{#5}\i) {}; }
        
        \IfInSetTF{#6}{(\i-1)*8+3} 
        { \node[Style full] at (c{#5}\i) {}; }
        { \node[Style empty] at (c{#5}\i) {}; }
        
        \IfInSetTF{#6}{(\i-1)*8+4} 
        { \node[Style full] at (d{#5}\i) {}; }
        { \node[Style empty] at (d{#5}\i) {}; }
        
        \IfInSetTF{#6}{(\i-1)*8+5} 
        { \node[Style full] at (e{#5}\i) {}; }
        { \node[Style empty] at (e{#5}\i) {}; }
        
        \IfInSetTF{#6}{(\i-1)*8+6} 
        { \node[Style full] at (f{#5}\i) {}; }
        { \node[Style empty] at (f{#5}\i) {}; }
        
        \IfInSetTF{#6}{(\i-1)*8+7} 
        { \node[Style full] at (g{#5}\i) {}; }
        { \node[Style empty] at (g{#5}\i) {}; }
        
        \IfInSetTF{#6}{(\i-1)*8+8} 
        { \node[Style full] at (h{#5}\i) {}; }
        {\node[Style empty] at (h{#5}\i) {}; }

        \IfInSetTF{#7}{(\i-1)*8+1} { \squareAround{a{#5}\i}; \node[left, xshift=-2pt] at (a{#5}\i) {$a_\i^{#5}$};} {}
        \IfInSetTF{#7}{(\i-1)*8+2} { \squareAround{b{#5}\i}; \node[right, xshift=3pt] at (b{#5}\i) {$b_\i^{#5}$};} {}
        \IfInSetTF{#7}{(\i-1)*8+6} { \squareAround{f{#5}\i}; \node[left, xshift=-2pt] at (f{#5}\i) {$f_\i^{#5}$};} {}
        \IfInSetTF{#7}{(\i-1)*8+7} { \squareAround{g{#5}\i}; \node[left, xshift=-2pt] at (g{#5}\i) {$g_\i^{#5}$};} {}
        \ifthenelse{\i=#1}{
            \IfInSetTF{#7}{(\i-1)*8+8} { \squareAround{h{#5}\i}; \node[right, xshift=3pt] at (h{#5}\i) {$h_\i^{#5}$};} {}
        }{
            \IfInSetTF{#7}{(\i-1)*8+8} { \squareAround{h{#5}\i}; \node[left, xshift=-2pt] at (h{#5}\i) {$h_\i^{#5}$};} {}
        }
    }

    \IfNoValueT{#8}
    {
        \node[font=\large\bfseries] at (#3 + #2*3*#1/2 - #2/2, #2*4 + #4){$#5$};
    }
}

\newcommand{\squareAround}[2][0.125]{%
  \draw[thick] ($(#2)+(-#1,-#1)$) rectangle ($(#2)+(#1,#1)$);%
}

\ExplSyntaxOn
\NewDocumentCommand{\IfInSetTF}{m m m m}{%
  \tl_clear:N \l_tmpa_tl
  \clist_map_inline:nn {#1}{
    \int_compare:nNnTF { \int_eval:n{#2} } = { \int_eval:n{##1} }
      { \tl_set:Nn \l_tmpa_tl {hit} \clist_map_break: }
      { }
  }
  \tl_if_eq:NnTF \l_tmpa_tl {hit} {#3}{#4}
}
\ExplSyntaxOff

\tikzset{My Style/.style={draw, circle, fill=black,scale=0.5}} 

\title{On $1$-limited and $(1,2)$-domination in cubic graphs}

\author{
Goran Radi\'c \\
\small \it University of Maribor, FNM, Koro\v ska cesta 160, 2000, Maribor, Slovenia
\\
\small \tt radic.goran@gmail.com  \\
\and
 Aleksandra Tepeh \\
\small \it University of Maribor, FEECS, Koro\v ska cesta 46, 2000, Maribor, Slovenia \\
\small \it 
Institute of Mathematics, Physics and Mechanics,  Jadranska ul.~19, 1000 Ljubljana, Slovenia \\
\small \it University of Maribor, FNM, Koro\v ska cesta 160, 2000, Maribor, Slovenia\\
\small \tt aleksandra.tepeh@um.si \\ 
}

\begin{document}
\maketitle

\begin{abstract}
A dominating set $D$ of a graph is called $1$-limited if every vertex
of $D$ has at most one neighbor outside $D$, while a $(1,2)$-dominating
set is a dominating set in which every vertex of the set has at least
two neighbors within the set. These two notions coincide on cubic
graphs. We prove that the decision problem
\textsc{1-Limited Dominating Set} is $\mathsf{NP}$-complete even when
restricted to $2$-connected planar cubic graphs, thereby completing the
known complexity results for \textsc{$k$-Limited Dominating Set} for
all fixed positive integers $k$. We also determine the exact
$1$-limited domination number of the entire Goldberg family. This
provides a further infinite family of cubic graphs supporting several
open conjectures and proposed bounds concerning $(1,2)$-domination and
induced cycles.
\end{abstract}


\medskip



\section{Introduction and Preliminaries}

Throughout the paper, all graphs are finite and simple. For a graph
$G$, we write $V(G)$ and $E(G)$ for its vertex and edge sets,
respectively. For a positive integer $k$, we write
$[k]=\{1,\ldots,k\}$.

Domination is one of the central and most extensively studied concepts
in graph theory. A \emph{dominating set} of a graph $G$ is a set
$D\subseteq V(G)$ such that every vertex outside $D$ has a neighbor
in $D$, and the minimum cardinality among all dominating sets of $G$ is the
\emph{domination number} $\gamma(G)$. Numerous variants of domination
have been introduced by imposing additional local requirements on the
dominating set or on the vertices outside it; see, for example,
\cite{HHH1,HHH2,HHH3}.

A recent addition to the family of domination parameters is
$k$-limited domination. Let $k$ be a positive integer. A dominating
set $D$ of a graph $G$ is called \emph{$k$-limited} if every vertex
of $D$ has at most $k$ neighbors outside $D$. The minimum cardinality
among all $k$-limited dominating sets of $G$ is the \emph{$k$-limited domination number}, denoted by
$\gamma_k^{\mathrm L}(G)$. The condition that every vertex of $D$ has
at most $k$ neighbors outside $D$ will be referred to as the
\emph{$k$-limited constraint}. The parameter was introduced and
systematically studied in \cite{podgorica1}, where general bounds,
structural properties, and connections with several classical
domination parameters were established. A subsequent study
\cite{San-samost} initiated the investigation of $k$-limited
domination in Cartesian products, establishing sharp general bounds
for $\gamma_k^{\mathrm L}(G\square H)$ and exact values for several
natural families of graph products.

One of the connections established in \cite{podgorica1} is
particularly useful for regular graphs. For a positive integer $t$
and a graph $G$ with minimum degree at least $t$, a
\emph{$(1,t)$-dominating set} is a dominating set $D$ such that every vertex
of $D$ has at least $t$ neighbors in $D$, and $\gamma_{1,t}(G)$
denotes the minimum cardinality of such a set. If $G$ is a connected
$d$-regular graph and $1\le k<d$, then
$\gamma_k^{\mathrm L}(G)=\gamma_{1,d-k}(G)$,
\cite{podgorica1}. Consequently, on cubic graphs,
\[
\gammaenaL(G)=\gamma_{1,2}(G).
\]

The $(1,2)$-domination number of cubic graphs was studied in
\cite{fakhran}, where the sharp bounds
$n/2\le\gamma_{1,2}(G)\le3n/4$
were established for every connected cubic graph $G$ of order $n$.
More recently, the relationship between $(1,2)$-domination and the
\emph{induced cycles vertex number} $\cind(G)$ was investigated in
\cite{TepehErves}. Here $\cind(G)$ is the maximum cardinality of a
vertex set inducing a $2$-regular subgraph. This study revealed a close
connection between the two parameters and led to several open
questions concerning cubic graphs. Exact results are also known for
particular structured families of cubic graphs; in particular, the
$(1,2)$-domination number of flower snarks was determined in
\cite{FLOWER}. This forms part of a broader recent interest in
domination-related parameters on snarks
\cite{Zhao,Maksimovic,Kratica,Burdet,Luiz,Pereira}.

Snarks are connected bridgeless cubic graphs that are not
$3$-edge-colorable and have played an important role in cubic graph
theory for many decades; see the survey~\cite{Bradley}. The Petersen
graph was the first known example, while the first explicitly
constructed infinite families were obtained by Isaacs~\cite{Isaacs}.
Shortly afterwards, Goldberg~\cite{Goldberg} introduced another
infinite family, now known as the Goldberg snarks. Their repetitive
block structure makes them a natural setting for studying local
domination conditions. Several domination parameters have already been
considered on this family, including connected and tree domination
\cite{Zhao}, Roman and independent Roman domination \cite{Luiz}, and
domination and independent domination \cite{Pereira}.

One of the main goals of the present paper is to determine the exact
$1$-limited domination number of this family; see
Theorem~\ref{glavnithm}. The main difficulty lies in the lower bound,
which is obtained through a local analysis of block weights together
with global compensation arguments. Our exact result also provides
further evidence for several open problems and conjectured bounds
concerning cubic graphs, which we discuss in more detail in the
concluding section.

We also address the computational complexity of $1$-limited
domination. It was recently proved that, for every fixed integer
$k\ge2$, \textsc{$k$-Limited Dominating Set} is $\mathsf{NP}$-complete
\cite{San-samost}. We settle the remaining case $k=1$ by showing that
\textsc{1-Limited Dominating Set} is $\mathsf{NP}$-complete even when restricted
to $2$-connected planar cubic graphs. Thus, together with the known
results for $k\ge2$, NP-completeness is established for every fixed
positive integer $k$. Since $1$-limited domination and
$(1,2)$-domination coincide on cubic graphs, the same hardness result
also holds for \textsc{$(1,2)$-Dominating Set}.

The paper is organized as follows. In Section~2 we establish the
NP-completeness of \textsc{1-Limited Dominating Set} for
$2$-connected planar cubic graphs. In Section~3 we introduce Goldberg
snarks and determine their exact $1$-limited, equivalently
$(1,2)$-domination, number. We conclude by discussing how these exact
values relate to several open problems on $(1,2)$-domination and
induced cycles in cubic graphs.

\section{Computational Complexity}

In this section, we investigate the computational complexity of
$1$-limited domination. We begin by formulating the corresponding
decision problem.

\begin{table}[h!]
\centering
\begin{tabular}{|l|}
\hline
\textsc{Instance:} A graph $G$ and a positive integer $k$.\\
\hline
\textsc{Question:} Does $G$ admit a $1$-limited dominating set $D$
with $|D|\le k$?\\
\hline
\end{tabular}
\caption{\textsc{1-Limited Dominating Set}}
\label{table:1limited}
\end{table}

Given a set $D\subseteq V(G)$, one can verify in polynomial time
whether $D$ dominates $G$ and whether every vertex of $D$ has at most
one neighbor outside $D$. Hence \textsc{1-Limited Dominating Set} belongs
to $\mathsf{NP}$. We will show that this decision problem is actually $\mathsf{NP}$-complete. To this end, it remains to prove that it is $\mathsf{NP}$-hard.
\medskip

To prove $\mathsf{NP}$-hardness, we give a polynomial reduction from
\textsc{Vertex Cover}, which remains $\mathsf{NP}$-hard even when
restricted to $2$-connected planar cubic graphs~\cite[Theorem~4.1(a)]{mohar2001}.
Recall that a \emph{vertex cover} of a graph $G$ is a set
$C\subseteq V(G)$ such that every edge of $G$ has at least one
endvertex in $C$.

\begin{table}[h!]
\centering
\begin{tabular}{|l|}
\hline
\textsc{Instance:} A graph $G$ and a positive integer $k$.\\
\hline
\textsc{Question:} Does $G$ admit a vertex cover $C$ with
$|C|\le k$?\\
\hline
\end{tabular}
\caption{\textsc{Vertex Cover}}
\label{table:VC}
\end{table}

Let $(G,k)$ be an instance of \textsc{Vertex Cover}, where
$V(G)=\{v_1,\dots,v_n\}$ and $m=|E(G)|$.
We construct a graph $G'$ by replacing each edge of $G$ with a
constant-size gadget. More precisely, for every edge
$v_iv_j\in E(G)$, we remove the edge $v_iv_j$ and insert a gadget
$Q^{ij}$ connecting $v_i$ and $v_j$, defined as follows.

We first introduce two triangles
$a^{ij}c^{ij}d^{ij}a^{ij}$ and 
$b^{ij}e^{ij}f^{ij}b^{ij}$,
which we call the \emph{left triangle} and the \emph{right triangle},
respectively. We connect them by adding the edge $d^{ij}e^{ij}$ and
the path
$c^{ij}g^{ij}h^{ij}f^{ij}$.
The subgraph induced by
$\{a^{ij},b^{ij},c^{ij},d^{ij},e^{ij},f^{ij},
  g^{ij},h^{ij}\}$
will be called the \emph{base part} of $Q^{ij}$.
Next, we construct a diamond on the vertex set
$\{i^{ij},j^{ij},k^{ij},\ell^{ij}\}$.
More precisely, we take the $4$-cycle
$i^{ij}k^{ij}j^{ij}\ell^{ij}i^{ij}$
and add the edge $k^{ij}\ell^{ij}$. 
Finally, we connect the diamond to the base part by adding the edges
$i^{ij}g^{ij}$ and $j^{ij}h^{ij}$, and connect the gadget to the
original vertices by adding the edges $v_ia^{ij}$ and $v_jb^{ij}$.
The resulting gadget is shown in Figure~\ref{fig:Qij}.

\begin{figure}[ht!]
\begin{center}
\begin{tikzpicture}[scale=0.9]

\node [My Style, name=vi, label=below:$v_i$] at (-4,0) {};
\node [My Style, name=vj, label=below:$v_j$] at (7,0) {};

\node [My Style, name=a, label=below:$a^{ij}$] at (-2.5,0) {};
\node [My Style, name=c, label=above:$c^{ij}$] at (-1.6,0.8) {};
\node [My Style, name=d, label=below:$d^{ij}$] at (-1.6,-0.8) {};

\node [My Style, name=g, label=below:$g^{ij}$] at (0,0.8) {};
\node [My Style, name=h, label=below:$h^{ij}$] at (3,0.8) {};

\node [My Style, name=e, label=below:$e^{ij}$] at (4.6,-0.8) {};
\node [My Style, name=f, label=above:$f^{ij}$] at (4.6,0.8) {};
\node [My Style, name=b, label=below:$b^{ij}$] at (5.5,0) {};

\node [My Style, name=i, label=left:$i^{ij}$] at (0,2.6) {};
\node [My Style, name=j, label=right:$j^{ij}$] at (3,2.6) {};
\node [My Style, name=l, label=below:$\ell^{ij}$] at (1.5,1.8) {};
\node [My Style, name=k, label=above:$k^{ij}$] at (1.5,3.6) {};

\draw[thick] (vi)--(a);
\draw[thick] (b)--(vj);

\draw[thick] (a)--(c);
\draw[thick] (a)--(d);
\draw[thick] (c)--(d);

\draw[thick] (b)--(e);
\draw[thick] (b)--(f);
\draw[thick] (e)--(f);

\draw[thick] (c)--(g);
\draw[thick] (g)--(h);
\draw[thick] (h)--(f);
\draw[thick] (d)--(e);

\draw[thick] (g)--(i);
\draw[thick] (h)--(j);
\draw[thick] (i)--(k);
\draw[thick] (j)--(k);
\draw[thick] (i)--(l);
\draw[thick] (j)--(l);
\draw[thick] (k)--(l);

\draw[rounded corners=12pt, thick] (-3.1,-1.5) rectangle (6.1,4.3);
\node[font=\large\bfseries] at (1.5,4.7) {$Q^{ij}$};

\end{tikzpicture}
\end{center}
\caption{The gadget $Q^{ij}$ replacing the edge $v_iv_j$.}
\label{fig:Qij}
\end{figure}
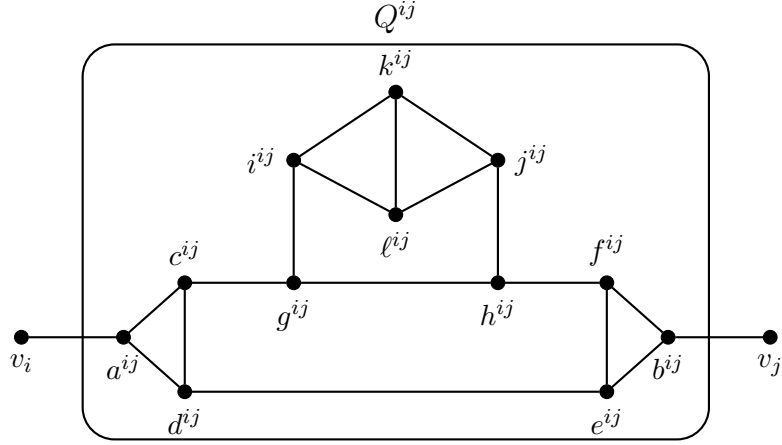

We denote the resulting graph by $G'$. Since every edge of $G$ is
replaced by a gadget containing $12$ new vertices,
$|V(G')|=n+12m$. The construction can clearly be carried out in polynomial time.

The vertices of $G$ remain vertices of $G'$ and will be called the
\emph{original vertices}, and all other vertices will be called
\emph{gadget vertices}. Every original vertex has the same degree in
$G'$ as in $G$, since each incident edge is replaced by exactly one
edge joining it to the corresponding gadget. Moreover, every gadget
vertex has degree three. Consequently, if $G$ is cubic, then $G'$ is
cubic.
The construction clearly preserves planarity and, since each edge is
replaced by a $2$-connected gadget attached at its endvertices,
also preserves $2$-connectivity.

The following proposition establishes the correctness of the
reduction.

\begin{proposition}
\label{NP}
Let $G'$ be the graph obtained from a graph $G$ with
$m=|E(G)|$ by replacing each edge $v_iv_j$ with the gadget
$Q^{ij}$ as described above. Then $G$ has a vertex cover of cardinality at most $k$ if
and only if $G'$ has a $1$-limited dominating set of cardinality at
most $k+8m$.
\end{proposition}

\begin{proof}
Let $C$ be a vertex cover of $G$ with $|C|\le k$. We regard $C$ as a
subset of $V(G')$. For each edge $v_iv_j\in E(G)$, fix one of its endvertices that
belongs to $C$. If the
chosen endvertex is $v_j$, let
\[
D^{ij}=
\{a^{ij},c^{ij},d^{ij},e^{ij},b^{ij},
  j^{ij},k^{ij},\ell^{ij}\},
\]
and if the chosen endvertex is $v_i$, let, symmetrically,
\[
D^{ij}=
\{b^{ij},e^{ij},f^{ij},d^{ij},a^{ij},
  i^{ij},k^{ij},\ell^{ij}\}.
\]
Set
\[
D=C\cup\bigcup_{v_iv_j\in E(G)}D^{ij}.
\]
Each set $D^{ij}$ contains exactly eight gadget vertices, and hence
$|D|=|C|+8m\le k+8m$.

A direct inspection shows that $D^{ij}$ dominates every vertex of
$Q^{ij}$ and that each vertex of $D^{ij}$ has at most one neighbor
outside $D$. Moreover, both $a^{ij}$ and $b^{ij}$ belong to
$D^{ij}$. Thus, every original vertex in $C$ has all its neighbors in
$D$. If $v_i\notin C$, then for each edge $v_iv_j$ incident with
$v_i$ we have $v_j\in C$, so $v_j$ is the chosen endvertex and
$a^{ij}\in D^{ij}$. Hence $v_i$ is dominated as well. Therefore $D$ is a
$1$-limited dominating set of $G'$ with $|D|\le k+8m$.

\medskip

Conversely, let $D$ be a $1$-limited dominating set of $G'$ with
$|D|\le k+8m$.
We first record two simple observations. The $1$-limited constraint implies that if a vertex of a triangle
belongs to $D$, then at least one other vertex of that triangle also
belongs to $D$.
Moreover, every diamond contains at least $3$ vertices of $D$.
Indeed, if $k^{ij}$ or $\ell^{ij}$ belongs to $D$, then at least two
of its three neighbors in the diamond must also belong to $D$.
If neither belongs to $D$, then domination of $k^{ij}$ and
$\ell^{ij}$ requires one of $i^{ij},j^{ij}$ to belong to $D$, but
such a vertex would have both $k^{ij}$ and $\ell^{ij}$ outside $D$,
a contradiction. Thus every diamond contains at least $3$ vertices of $D$.

We claim that every gadget contains at least $8$ vertices of
$D$. Suppose, to the contrary, that
$|D\cap V(Q^{ij})|\le7$. For convenience, set
$L^{ij}=\{a^{ij},c^{ij},d^{ij},g^{ij}\}$ and 
$R^{ij}=\{b^{ij},e^{ij},f^{ij},h^{ij}\}$. Both $L^{ij}$ and $R^{ij}$ contain a vertex
of $D$, since otherwise $c^{ij}$ or $f^{ij}$, respectively, would be
undominated.

Neither set can contain exactly one vertex of $D$. Indeed, if
$|D\cap L^{ij}|=1$, then the triangle observation implies that this
vertex is $g^{ij}$. Thus $a^{ij},c^{ij},d^{ij}\notin D$. Since
$d^{ij}$ must be dominated, $e^{ij}\in D$, and the $1$-limited
constraint at $e^{ij}$ forces $b^{ij},f^{ij}\in D$. As the diamond
already contributes at least three vertices, the assumption
$|D\cap V(Q^{ij})|\le7$ implies $h^{ij}\notin D$. But then
$g^{ij}$ has two neighbors outside $D$, a
contradiction. The argument for $R^{ij}$ is symmetric.

Consequently,
$|D\cap L^{ij}|\ge2$
and
$|D\cap R^{ij}|\ge2$.
Since the diamond contains at least three vertices of $D$ and
$|D\cap V(Q^{ij})|\le7$ we  derive
that 
$|D\cap L^{ij}|=|D\cap R^{ij}|=2$,
and the diamond contains exactly three vertices of $D$. The triangle
observation now implies that $g^{ij},h^{ij}\notin D$. If the only vertex of the diamond outside $D$ were $k^{ij}$ or
$\ell^{ij}$, then $i^{ij}\in D$ would have two neighbors outside $D$.
Hence the unique vertex of the diamond outside $D$ is one of
$i^{ij},j^{ij}$; by symmetry, assume that $i^{ij}\notin D$. Since
$g^{ij},h^{ij},i^{ij}\notin D$, domination of $g^{ij}$ forces
$c^{ij}\in D$, and the $1$-limited constraint at $c^{ij}$ then forces
$a^{ij},d^{ij}\in D$, contradicting
$|D\cap L^{ij}|=2$. Thus
$|D\cap V(Q^{ij})|\ge8$.

We next show that a gadget contributes at least $9$ vertices whenever
$v_i,v_j\notin D$. We claim that in this case
$|D\cap L^{ij}|\ge3$. If $a^{ij}\in D$, then, since
$v_i\notin D$, the $1$-limited constraint at $a^{ij}$ forces
$c^{ij},d^{ij}\in D$. If $a^{ij}\notin D$, then its domination
requires one of $c^{ij},d^{ij}$ to belong to $D$. The triangle
observation forces the other one to belong to $D$ as well, and the
$1$-limited constraint at $c^{ij}$ then forces $g^{ij}\in D$.
Thus $|D\cap L^{ij}|\ge3$. Symmetrically,
$|D\cap R^{ij}|\ge3$. Together with the three vertices forced in the
diamond, this gives
$|D\cap V(Q^{ij})|\ge9$.

Let
$C=D\cap V(G)$
and let
$F=\{v_iv_j\in E(G):v_i,v_j\notin C\}$
be the set of edges not covered by $C$. Since every gadget contributes
at least $8$ vertices to $D$, and every edge of $F$ forces one
additional vertex in its gadget, we have
$|D|\ge |C|+8m+|F|$.
As $|D|\le k+8m$, it follows that
$|C|+|F|\le k$.

For each edge of $F$, choose one of its endvertices, and let $X$ be
the set of chosen vertices. Every edge outside $F$ is covered by $C$,
while every edge in $F$ is covered by $X$. Hence $C\cup X$ is a
vertex cover of $G$. Moreover, $C$ and $X$ are disjoint, and
$|X|\le|F|$, so
\[
|C\cup X|=|C|+|X|\le |C|+|F|\le k.
\]
Thus $G$ has a vertex cover of cardinality at most $k$.
\end{proof}

Proposition~\ref{NP} establishes the correctness of the reduction.
The construction is polynomial and maps $2$-connected planar cubic
graphs to $2$-connected planar cubic graphs. Since \textsc{Vertex Cover} is $\mathsf{NP}$-hard even on
$2$-connected planar cubic graphs~\cite{mohar2001}, while
\textsc{1-Limited Dominating Set} belongs to $\mathsf{NP}$, we obtain
the following result.

\begin{theorem}
\label{thm:NP-complete}
The decision problem \textsc{1-Limited Dominating Set} is
$\mathsf{NP}$-complete even when restricted to $2$-connected planar
cubic graphs.
\end{theorem}

The corresponding decision problem \textsc{$(1,2)$-Dominating Set}
asks whether a given graph $G$ has a $(1,2)$-dominating set of
cardinality at most $k$. As observed in the introduction, $1$-limited
domination and $(1,2)$-domination coincide on cubic graphs. Hence
Theorem~\ref{thm:NP-complete} immediately implies the following.

\begin{corollary}
The decision problem
\textsc{$(1,2)$-Dominating Set} is $\mathsf{NP}$-complete even when
restricted to $2$-connected planar cubic graphs.
\end{corollary}

\section{Goldberg snarks}

We first describe the family of Goldberg snarks, which will be the
main object of the remainder of the paper. The construction is based
on the fixed graph $B$, called the \emph{Goldberg block}, or simply a \emph{block}, shown in
Figure~\ref{fig:goldberg-basic}. It has vertex set
$V(B)=\{a,b,c,d,e,f,g,h\}$ 
and edge set
$E(B)=\{ab,ae,bc,cd,cf,de,dh,eg,fg\}$.
For $k\ge3$, let $B_1,\ldots,B_k$ be pairwise disjoint copies of
$B$, with
$V(B_i)=\{a_i,b_i,c_i,d_i,e_i,f_i,g_i,h_i\}$.

The graph $G_k$ is obtained by joining these copies cyclically: for
every $i\in[k]$, where indices are taken modulo $k$, we add the edges
$b_i a_{i+1}$, $g_i f_{i+1}$ and $h_i h_{i+1}$.
The resulting graph is connected, bridgeless, cubic, and has $8k$
vertices.

\begin{figure}[!ht]
\centering
\begin{tikzpicture}[scale=0.85, every node/.style={font=\small}]
\tikzset{v/.style={draw, circle, fill=black, inner sep=1.5pt}}

\begin{scope}
   \coordinate (aL) at (0,3);
   \coordinate (bL) at (2,3);
   \coordinate (cL) at (0,2);
   \coordinate (dL) at (1,2);
   \coordinate (eL) at (2,2);
   \coordinate (fL) at (0,1);
   \coordinate (gL) at (2,1);
   \coordinate (hL) at (1,0);

   \draw (aL) -- (bL);
   \draw (aL) -- (eL);
   \draw (bL) -- (cL);
   \draw (cL) -- (dL);
   \draw (cL) -- (fL);
   \draw (dL) -- (eL);
   \draw (dL) -- (hL);
   \draw (eL) -- (gL);
   \draw (fL) -- (gL);

   \node[v,label=above:$a$] at (aL) {};
   \node[v,label=above:$b$] at (bL) {};
   \node[v,label=left:$c$] at (cL) {};
   \node[v,label=below right:$d$] at (dL) {};
   \node[v,label=right:$e$] at (eL) {};
   \node[v,label=below:$f$] at (fL) {};
   \node[v,label=below:$g$] at (gL) {};
   \node[v,label=below:$h$] at (hL) {};

   \node at (1,-0.95) {$B$};
\end{scope}

\begin{scope}[xshift=5.2cm]
   \foreach \name/\x/\lab in {I/0/$B_i$,J/3.7/$B_{i+1}$} {
      \coordinate (a\name) at (\x,3);
      \coordinate (b\name) at (\x+2,3);
      \coordinate (c\name) at (\x,2);
      \coordinate (d\name) at (\x+1,2);
      \coordinate (e\name) at (\x+2,2);
      \coordinate (f\name) at (\x,1);
      \coordinate (g\name) at (\x+2,1);
      \coordinate (h\name) at (\x+1,0);

      \draw (a\name) -- (b\name);
      \draw (a\name) -- (e\name);
      \draw (b\name) -- (c\name);
      \draw (c\name) -- (d\name);
      \draw (c\name) -- (f\name);
      \draw (d\name) -- (e\name);
      \draw (d\name) -- (h\name);
      \draw (e\name) -- (g\name);
      \draw (f\name) -- (g\name);

      \foreach \xv in {a,b,c,d,e,f,g,h} {
         \node[v] at (\xv\name) {};
      }

      \node at (\x+1,-0.95) {\lab};
   }

   \draw[dashed] (bI) -- (aJ);
   \draw[dashed] (gI) -- (fJ);
   \draw[dashed] (hI) -- (hJ);

   \node[above] at (bI) {$b_i$};
   \node[above] at (aJ) {$a_{i+1}$};
   \node[below] at (gI) {$g_i$};
   \node[below] at (fJ) {$f_{i+1}$};
   \node[below left] at (hI) {$h_i$};
   \node[below right] at (hJ) {$h_{i+1}$};
\end{scope}

\end{tikzpicture}
\caption{The Goldberg block $B$ and the edges joining two consecutive
copies.}
\label{fig:goldberg-basic}
\end{figure}
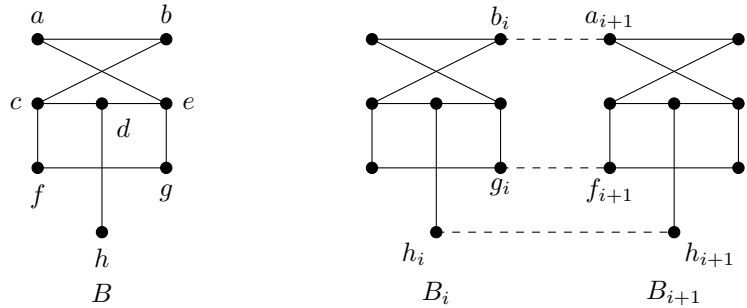

For odd $k\ge5$, the graph $G_k$ is a Goldberg snark
\cite{Goldberg}. Nevertheless, since our arguments apply to every
$k\ge3$, we will, with a slight abuse of terminology, refer to all
graphs $G_k$, $k\ge3$, as Goldberg snarks.

For a set $D\subseteq V(G_k)$, let
$w_i=|D\cap V(B_i)|$ for $i\in[k]$.
The cyclic sequence $(w_1,\ldots,w_k)$ is called the
\emph{weight sequence} of $D$. A consecutive subsequence of a weight
sequence will be referred to as a \emph{pattern}. For brevity,
we write patterns without parentheses and commas; for example,
$(4,4,5)$ is written as $445$. All indices are taken modulo $k$, so in
particular $w_0=w_k$ and $w_{k+1}=w_1$.

\medskip

We first establish an upper bound on $\gammaenaL(G_k)$ by
constructing explicit $1$-limited dominating sets. The construction
uses three \emph{configurations}, denoted by $X$, $Y$, and $Z$, which
consist of $3$, $4$, and $5$ consecutive Goldberg blocks,
respectively, see Figure~\ref{XYZ}.

For $P\in\{X,Y,Z\}$, denote by $D_P$ the set of black vertices in
$P$. All vertices of $D_P$ satisfy the $1$-limited constraint within
the configuration, except for the two square boundary vertices, whose
constraint is completed through edges to neighboring configurations.
From Figure~\ref{XYZ}, we further observe that
$|D_X|=13$, $|D_Y|=18$ and $|D_Z|=23$.

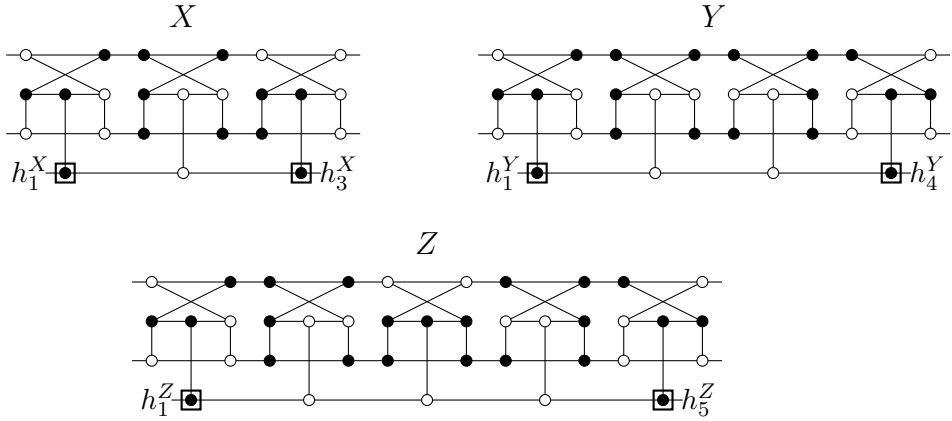
\begin{figure}[!ht]
\centering

\begin{tikzpicture}[scale=1]

\def\dis{0.52}

\Goldbergsnark
{3}
{\dis}
{0}
{4.8}
{X}
[2,3,4,8,9,10,11,14,15,19,20,22,24]
[8,24]

\Goldbergsnark
{4}
{\dis}
{12*\dis}
{4.8}
{Y}
[2,3,4,8,9,10,11,14,15,17,18,21,22,23,25,28,29,32]
[8,32]

\Goldbergsnark
{5}
{\dis}
{3.2*\dis}
{1.8}
{Z}
[2,3,4,8,9,10,11,14,15,19,20,21,22,23,25,26,29,30,31,33,36,37,40]
[8,40]

\end{tikzpicture}

\caption{The configurations $X$, $Y$, and $Z$.}

\label{XYZ}
\end{figure}

We are now ready to establish the required upper bound.

\begin{lemma}\label{upper}
Let $k\ge3$. Then

\[
\gammaenaL(G_k)\leq
\begin{cases}
\frac{13}{3}k; & k\equiv0\pmod3\\
\frac{13}{3}k+\frac23; & k\equiv1\pmod3\\
\frac{13}{3}k+\frac43; & k\equiv2\pmod3.
\end{cases}
\]
Moreover, 
$\gammaenaL(G_5)\le 22$.
\end{lemma}

\begin{proof}
A direct inspection of Figure~\ref{XYZ} shows that in each
configuration $P\in\{X,Y,Z\}$ every vertex outside $D_P$ has a
neighbor in $D_P$, and every vertex of $D_P$ has at most one neighbor
outside $D_P$, except for the two square boundary vertices.

When two configurations are placed consecutively, their boundary
blocks are joined by the three edges prescribed in the construction
of $G_k$. We refer to this operation as concatenation. A concatenation
is cyclic if the last configuration is joined to the first in the
same way. In such a concatenation, the corresponding square vertices
are adjacent and complete each other's $1$-limited constraint.
Therefore every cyclic concatenation of $X$, $Y$, and $Z$ determines
a $1$-limited dominating set.

First, let $k=3q$, where $q\ge1$. By cyclically concatenating $q$ copies
of $X$, we obtain a $1$-limited dominating set of $G_k$ of cardinality
   $$q|D_X|=13q=\tfrac{13}{3}k.$$
Next, let $k=3q+1$, where $q\ge1$. The cyclic concatenation
$YX^{q-1}$ consists of
$4+3(q-1)=3q+1=k$
Goldberg blocks and determines a $1$-limited dominating set of size
$$|D_Y|+(q-1)|D_X|
   =18+13(q-1)
   =13q+5
   =\tfrac{13}{3}k+\tfrac23.$$
Finally, let $k=3q+2$, where $q\ge1$. The cyclic concatenation
$ZX^{q-1}$ consists of
$5+3(q-1)=3q+2=k$
Goldberg blocks and determines a $1$-limited dominating set of size
   $$|D_Z|+(q-1)|D_X|
   =23+13(q-1)
   =13q+10
   =\tfrac{13}{3}k+\tfrac43.$$
This proves the general upper bound.

For $k=5$, the configuration $Z$ gives a set of cardinality $23$. The
construction in Figure~\ref{G5} has weight sequence $44644$
and contains $22$ vertices. The black vertices form a $1$-limited
dominating set, and hence $\gammaenaL(G_5)\le22$.
\end{proof}

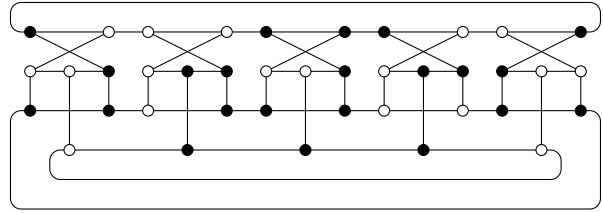
\begin{figure}[!ht]
\begin{center}
\begin{tikzpicture}[scale=1]

\def\dis{0.52}
\def\x{0} 
\def\y{10} 

\Goldbergsnark{5}{\dis}{\x*\dis}{\y-\dis*6}{}[1,5,6,7,12,13,15,16,17,18,21,22,23,24,25,28,29,32,34,35,38,39][-1][1]

\end{tikzpicture}
\end{center}
    \caption{A $1$-limited dominating set of $G_5$ with weight sequence
$44644$.}
    \label{G5}
\end{figure}

We now turn to the lower bound. Combined with the constructions above,
it will determine the exact value of $\gammaenaL(G_k)$. In the local
arguments below, we will repeatedly use two symmetries of $G_k$. The
first interchanges
\[
a_i\leftrightarrow f_i,\qquad
b_i\leftrightarrow g_i,\qquad
c_i\leftrightarrow e_i
\]
for every $i\in[k]$, while fixing $d_i$ and $h_i$. The second reverses
the cyclic order of the blocks, interchanging the left and right sides
of each block. 

Throughout this section, let $k\ge3$ and let $D$ be a minimum $1$-limited dominating
set of $G_k$. Since
$|D|=\sum_{i=1}^{k}w_i$,
a lower bound on $|D|$ can be obtained by deriving suitable lower
bounds on the total weight of consecutive blocks.

We begin with the following basic lower bound on the weight of a single block.

\begin{lemma}    
\label{wi3}
For every $i\in[k]$, we have $w_i\ge3$.
\end{lemma}

\begin{proof}
Fix $i\in[k]$. If $d_i\in D$, then, since $D$ is $1$-limited, at least
two of $c_i,e_i,h_i$ belong to $D$, and hence $w_i\ge3$.
Suppose now that $d_i\notin D$. Since $D$ is dominating, at least one of
$c_i,e_i,h_i$ belongs to $D$. If $c_i\in D$, then both $b_i$ and $f_i$
belong to $D$, and hence $w_i\ge3$. Similarly, if $e_i\in D$, then both
$a_i$ and $g_i$ belong to $D$, and again $w_i\ge3$.
It remains to consider the case where $h_i\in D$ and
$c_i,e_i\notin D$. Since $c_i$ and $e_i$ must be dominated and
$d_i\notin D$, at least one of $b_i,f_i$ and at least one of $a_i,g_i$
belong to $D$. Together with $h_i\in D$, this gives $w_i\ge3$.
\end{proof}

Although every block has weight at least three, equality is highly
restrictive. The next lemma characterizes blocks of weight
$3$ and shows that each of their neighboring blocks has weight at
least $6$.

\begin{lemma}
\label{weight3}
If $w_i=3$, then
$w_{i-1}\ge6$ and $w_{i+1}\ge6$. Moreover, if $w_{i-1}=6$, then $d_{i-1}\notin D$, and if
$w_{i+1}=6$, then $d_{i+1}\notin D$. 
\end{lemma}

\begin{proof}
Suppose that $w_i=3$. We first show that $c_i,e_i\notin D$. Assume
that $c_i\in D$. Since $D$ is $1$-limited, at least two vertices of
$\{b_i,d_i,f_i\}$ belong to $D$. If $d_i\in D$, then at least one of
$e_i,h_i$ also belongs to $D$, giving $w_i\ge4$. If $d_i\notin D$,
then $b_i,f_i\in D$, and since $w_i=3$, the vertex $e_i$ is
undominated. Both cases are contradictions. Thus $c_i\notin D$, and
similarly $e_i\notin D$.

It follows that $d_i\notin D$ and hence $h_i\in D$. The vertex $c_i$
must be dominated by $b_i$ or $f_i$. They cannot both belong to $D$,
since then $a_i,g_i\in D$, contradicting $w_i=3$. Hence exactly one
of $b_i,f_i$ belongs to $D$. If $b_i\in D$, then $a_i\in D$, while
if $f_i\in D$, then $g_i\in D$. Therefore
$D\cap V(B_i)=\{a_i,b_i,h_i\}$
or
$D\cap V(B_i)=\{f_i,g_i,h_i\}$.

We consider the first case, and the second is symmetric. Since
$D\cap V(B_i)=\{a_i,b_i,h_i\}$, we obtain
$\{a_{i+1},c_{i+1},f_{i+1},g_{i+1},h_{i+1}\}\subseteq D$.
The $1$-limited constraints at $a_{i+1}$ and $c_{i+1}$ require,
respectively, at least one of $\{b_{i+1},e_{i+1}\}$ and at least one
of $\{b_{i+1},d_{i+1}\}$ to belong to $D$. Hence $w_{i+1}\ge6$.
If $w_{i+1}=6$, the unique additional vertex is $b_{i+1}$, and thus
$d_{i+1}\notin D$.
By symmetry, $w_{i-1}\ge6$, and if $w_{i-1}=6$, then
$d_{i-1}\notin D$.
\end{proof}

Thus a block of weight $3$ is strongly compensated by its
neighbors. Hence low average weight can only persist through stretches
of blocks of weight $4$.

\begin{lemma}
\label{weight4}
Let $i\in[k]$. Then the following statements hold.
\begin{enumerate}
    \item[\rm(i)] If $w_i=4$, then
    $d_i\in D$ if and only if $h_i\in D$.

    \item[\rm(ii)] If $w_i=4$ and $d_i,h_i\in D$, then one of the
following holds:
\begin{enumerate}
    \item[(a)]
    $\{a_{i+1},c_{i+1},f_{i+1},g_{i+1}\}\subseteq D$ and
    at least one of $b_{i+1},e_{i+1}$ belongs to $D$;
    \item[(b)]
    $\{a_{i-1},b_{i-1},c_{i-1},g_{i-1}\}\subseteq D$ and
    at least one of $e_{i-1},f_{i-1}$ belongs to $D$.
\end{enumerate}
In particular, $w_{i-1}\ge5$ or $w_{i+1}\ge5$.

    \item[\rm(iii)] If $w_iw_{i+1}=44$, $d_i\notin D$, and
    $d_{i+1}\in D$, then $w_{i+2}\ge6$. Symmetrically, if
    $d_i\in D$ and $d_{i+1}\notin D$, then $w_{i-1}\ge6$.

    \item[\rm(iv)] If $w_iw_{i+1}w_{i+2}=444$, then
    $d_{i+1},h_{i+1}\notin D$, and
    $w_{i-1}\ge6$ or $w_{i+3}\ge6$.
\end{enumerate}
\end{lemma}

\begin{proof}
Assume that $w_i=4$. If $c_i,e_i\notin D$, then $d_i\notin D$ and
$h_i\in D$. Moreover, $a_i,b_i$ either both belong to $D$ or both lie
outside $D$, and the same holds for $f_i,g_i$. Hence
$w_i\in\{3,5\}$, a contradiction. If $c_i,e_i\in D$, then the
$1$-limited constraint forces at least three additional vertices of
$B_i$ into $D$, so $w_i\ge5$, again a contradiction. Consequently,
exactly one of $c_i,e_i$ belongs to $D$.

\smallskip
\noindent
(i)\,
By symmetry, let $c_i\in D$ and $e_i\notin D$. If $d_i\in D$, then
$h_i\in D$. If $d_i\notin D$, then $b_i,f_i\in D$, and $h_i\in D$
would leave $e_i$ undominated. Hence $h_i\notin D$.

\smallskip
\noindent
(ii)\,
Suppose that $d_i,h_i\in D$ and, by symmetry, that
$c_i\in D$ and $e_i\notin D$. Then exactly one of $b_i,f_i$ belongs
to $D$.
If $b_i\in D$, then
$a_{i+1},f_{i+1},c_{i+1},g_{i+1}\in D$, and at least one of
$b_{i+1},e_{i+1}$ belongs to $D$. Hence (a) holds.
If $f_i\in D$, then $b_{i-1},g_{i-1}\in D$.
The $1$-limited condition at $b_{i-1}$ gives
$a_{i-1},c_{i-1}\in D$, while at least one of
$e_{i-1},f_{i-1}$ must belong to $D$. Hence (b) holds.

\smallskip
\noindent
(iii)\,
Suppose that $w_iw_{i+1}=44$, $d_i\notin D$, and $d_{i+1}\in D$.
By (i), $h_i\notin D$ and $h_{i+1}\in D$. Applying (ii) to
$B_{i+1}$, alternative (b) is impossible, since it would give
$w_i\ge5$. Thus (a) holds, so at least five vertices of $B_{i+2}$
belong to $D$. Since $h_i\notin D$ and $h_{i+1}\in D$, also
$h_{i+2}\in D$, and hence $w_{i+2}\ge6$. The second assertion follows
by symmetry.

\smallskip
\noindent
(iv)\,
Suppose that $w_iw_{i+1}w_{i+2}=444$. By (ii),
$d_{i+1},h_{i+1}$ cannot belong to $D$, and hence both are outside
$D$ by (i). Thus at least one of $h_i,h_{i+2}$ belongs to $D$.
If $h_{i+2}\in D$, then (i) and (iii), applied to
$B_{i+1},B_{i+2}$, give $w_{i+3}\ge6$. If $h_i\in D$, the symmetric
argument gives $w_{i-1}\ge6$.
\end{proof}

We now translate these local restrictions into lower bounds on the
total weight of consecutive blocks. For $i\in[k]$ and $n\ge1$, the
$n$ consecutive blocks $B_i,B_{i+1},\ldots,B_{i+n-1}$
will be called an \emph{$n$-block}, and its weight is
\[
W_i(n)=\sum_{j=0}^{n-1}w_{i+j},
\]
where indices are taken modulo $k$.
We first consider two and three consecutive blocks.

\begin{lemma}\label{8in12}
Let $i\in[k]$. Then
\begin{enumerate}
\item[\rm(i)]
$W_i(2)\ge8$, with equality if and only if
$w_iw_{i+1}=44$.

\item[\rm(ii)]
$W_i(3)\ge12$, with equality if and only if
$w_iw_{i+1}w_{i+2}\in\{363,444\}$.
\end{enumerate}
\end{lemma}

\begin{proof}
If one of two consecutive blocks has weight $3$, then its neighbor has
weight at least $6$ by Lemma~\ref{weight3}; otherwise both have weight
at least $4$. Hence $W_i(2)\ge8$, with equality precisely for the
pattern $44$.

For three consecutive blocks, if none has weight $3$, their total
weight is at least $12$, with equality precisely for $444$. If the
middle block has weight $3$, Lemma~\ref{weight3} gives total weight at
least $15$. If an end block has weight $3$, its neighbor has weight at
least $6$ and the remaining block at least $3$. Hence the total weight
is at least $12$, with equality precisely for $363$.
\end{proof}

The equality cases in Lemma~\ref{8in12}(ii) are themselves restrictive:
a $3$-block of minimum weight forces additional weight in the
neighboring $3$-blocks.

\begin{lemma}
\label{tight3windows}
Let $k\ge6$ and let $i\in[k]$. Suppose that $W_i(3)=12$. Then
\begin{enumerate}
\item[\rm(i)]
If $w_iw_{i+1}w_{i+2}=363$, then
$W_{i-3}(3)\ge14$ and $W_{i+3}(3)\ge14$.

\item[\rm(ii)]
If $w_iw_{i+1}w_{i+2}=444$, then
$W_{i-3}(3),W_{i+3}(3)\ge13$, and at least one of them is at least
$14$.
\end{enumerate}
\end{lemma}

\begin{proof}
If $w_iw_{i+1}w_{i+2}=363$, then Lemma~\ref{weight3} gives
$w_{i-1},w_{i+3}\ge6$. Hence, by Lemma~\ref{8in12}(i),
$W_{i-3}(3)\ge14$
and
$W_{i+3}(3)\ge14$.
Suppose now that $w_iw_{i+1}w_{i+2}=444$. By
Lemma~\ref{weight4}(iv), $w_{i-1}\ge6$ or $w_{i+3}\ge6$.
By symmetry, assume that $w_{i-1}\ge6$. Then
$W_{i-3}(3)\ge14$.

It remains to show that $W_{i+3}(3)\ge13$. If
$W_{i+3}(3)=12$, then by Lemma~\ref{8in12}(ii) its pattern is either
$363$ or $444$. The former is impossible by Lemma~\ref{weight3},
while the latter would give the pattern $444444$, contradicting
Lemma~\ref{weight4}(iv) applied to
$w_{i+2}w_{i+3}w_{i+4}=444$.
Thus $W_{i+3}(3)\ge13$. The other case follows by symmetry.
\end{proof}

These local restrictions already force the existence of a relatively
heavy $4$-block.

\begin{lemma}
\label{4blocks}
If $k\ge4$, then there exists $j\in[k]$ such that
$W_j(4)\ge18$.
\end{lemma}

\begin{proof}
If $w_i=3$ for some $i$, then Lemma~\ref{weight3} gives
$W_{i-1}(4)\ge6+3+6+3=18$.
Hence assume that every block has weight at least $4$.

If some block has weight at least $6$, the result follows immediately.
Thus suppose that every block has weight $4$ or $5$. If some $4$-block
contains two blocks of weight $5$, its weight is at least $18$.
Otherwise, either all blocks have weight $4$, or every block of weight
$5$ is followed by three blocks of weight $4$. Hence in either case
there are three consecutive blocks of weight $4$. By
Lemma~\ref{weight4}(iv), one of their neighboring blocks has weight at
least $6$, a contradiction.
\end{proof}

The next two lemmas extend these estimates to $5$- and $6$-blocks.

\begin{lemma}
\label{5blocks}
Let $k\ge5$ and let $i\in[k]$. Then $W_i(5)\ge21$.
Moreover, there exists $j\in[k]$ such that
$W_j(5)\ge22$.
\end{lemma}

\begin{proof}
By Lemma~\ref{8in12},
$W_i(5)=W_i(2)+W_{i+2}(3)\ge20$.
Suppose that equality holds. Then Lemma~\ref{8in12} gives
$W_i(2)=8$, and hence
$w_iw_{i+1}=44$. Using instead
$W_i(5)=W_i(3)+W_{i+3}(2)$ gives
$w_{i+3}w_{i+4}=44$. Hence $w_{i+2}=4$, and the resulting pattern
$44444$ contradicts Lemma~\ref{weight4}(iv). Therefore
$W_i(5)\ge21$.

By Lemma~\ref{4blocks}, choose $j\in[k]$ with $W_j(4)\ge18$.
If $w_{j+4}\ge4$, then $W_j(5)\ge22$. Otherwise $w_{j+4}=3$, and
Lemma~\ref{weight3} gives $w_{j+3},w_{j+5}\ge6$. Hence
$W_{j+3}(3)\ge15$, and Lemma~\ref{8in12}(i) gives
$W_{j+3}(5)
=
W_{j+3}(3)+W_{j+6}(2)
\ge15+8=23$.
Thus some $5$-block has weight at least $22$.
\end{proof}

\begin{lemma}
\label{6blocks}
If $k\ge6$, then $W_i(6)\ge25$ for every $i\in[k]$ and there exists
$j\in[k]$ such that $W_j(6)\ge26$.
\end{lemma}

\begin{proof}
Suppose that $W_i(6)\le24$ for some $i\in[k]$. Since
$W_i(6)=W_i(3)+W_{i+3}(3)$ and both summands are at least $12$, we have
$W_i(3)=W_{i+3}(3)=12$. However, Lemma~\ref{tight3windows}, applied to
$W_i(3)=12$, gives $W_{i+3}(3)\ge13$, a contradiction.

If every $3$-block has weight at least $13$, then every $6$-block has
weight at least $26$. Otherwise $W_j(3)=12$ for some $j$. By
Lemma~\ref{tight3windows}, one of
$W_{j-3}(3)$ and $W_{j+3}(3)$ is at least $14$, and hence
$W_{j-3}(6)\ge26$ or $W_j(6)\ge26$.
\end{proof}

We will need one further local restriction on blocks of weight $5$.
It shows that a block containing $h_i$ cannot have both neighboring
blocks of weight $4$.

\begin{lemma}
\label{w=5}
Let $i\in[k]$. If $w_i=5$ and $h_i\in D$, then
$w_{i-1}\ge5$ or $w_{i+1}\ge5$.
\end{lemma}

\begin{proof}
Suppose, to the contrary, that $w_{i-1},w_{i+1}\le4$. By
Lemmas~\ref{wi3} and~\ref{weight3},
$w_{i-1}=w_{i+1}=4$.

Suppose first that $d_i\notin D$. Since $h_i\in D$, we have
$h_{i-1},h_{i+1}\in D$, and hence
$d_{i-1},d_{i+1}\in D$ by Lemma~\ref{weight4}(i).
Since $w_{i-1}=w_{i+1}=4$, neither
$\{b_{i-1},g_{i-1}\}$ nor $\{a_{i+1},f_{i+1}\}$ can be contained in
$D$, as either would force the corresponding weight to at least $5$.
Also, $c_i,e_i$ cannot both belong to $D$, as this would give
$w_i\ge7$. 

If neither
belongs to $D$, then
$D\cap V(B_i)=\{a_i,b_i,f_i,g_i,h_i\}$,
which forces $b_{i-1},g_{i-1}\in D$, a contradiction. Thus, by
symmetry, assume that $c_i\in D$ and $e_i\notin D$. Then
$b_i,f_i\in D$, and exactly one of $a_i,g_i$ belongs to $D$.
If $a_i\in D$, then $b_{i-1},g_{i-1}\in D$, while if
$g_i\in D$, then $a_{i+1},f_{i+1}\in D$, again a contradiction.
Hence $d_i\in D$.
Now we have $d_i,h_i\in D$, thus  exactly one of $c_i,e_i$ belongs to $D$. By
symmetry, assume that $c_i\in D$ and $e_i\notin D$.

Suppose first that $f_i\notin D$. Then $g_i\notin D$, and hence
$D\cap V(B_i)=\{a_i,b_i,c_i,d_i,h_i\}$.
It follows that $f_{i+1},c_{i+1},g_{i+1}\in D$. Since
$w_{i+1}=4$, Lemma~\ref{weight4}(i) gives
$d_{i+1},h_{i+1}\notin D$. Hence
$h_{i-1},d_{i-1},b_{i-1}\in D$. The fourth vertex of
$D\cap V(B_{i-1})$ must be $c_{i-1}$, since $e_{i-1}\in D$ would
violate the $1$-limited constraint at $e_{i-1}$. But then
$g_{i-1}$ is undominated, a contradiction. Thus  $f_i\in D$.

Suppose next that $b_i\notin D$. Then $D\cap V(B_i)=\{c_i,d_i,f_i,g_i,h_i\}$.
It follows that $b_{i-1},a_{i-1},c_{i-1}\in D$. Since
$w_{i-1}=4$, Lemma~\ref{weight4}(i) gives
$d_{i-1},h_{i-1}\notin D$. Hence
$h_{i+1},d_{i+1}\in D$. As $g_i$ has $e_i\notin D$ as a neighbor, we also have $f_{i+1}\in D$, and 
the fourth vertex of
$D\cap V(B_{i+1})$ must be $c_{i+1}$. But then
$a_{i+1}$ is undominated, a contradiction.

Therefore, we have
$D\cap V(B_i)=\{b_i,c_i,d_i,f_i,h_i\}$.
Hence $a_{i+1},g_{i-1}\in D$. Since $h_i,d_i\in D$, at least one of
$h_{i-1},h_{i+1}$ belongs to $D$.

If $h_{i+1}\in D$, then $d_{i+1}\in D$. Since $w_{i+1}=4$, exactly
one of $c_{i+1},e_{i+1}$ belongs to $D$. If $c_{i+1}\in D$, it has
two neighbors outside $D$, while if $e_{i+1}\in D$, the vertex
$f_{i+1}$ is undominated. Thus $h_{i+1}\notin D$.
Hence $h_{i-1},d_{i-1}\in D$. Again, exactly one of
$c_{i-1},e_{i-1}$ belongs to $D$. If $c_{i-1}\in D$, it has two
neighbors outside $D$, while if $e_{i-1}\in D$, the vertex
$b_{i-1}$ is undominated. This final contradiction completes the
proof.
\end{proof}

We can now characterize the $6$-blocks attaining the minimum weight $25$
from Lemma~\ref{6blocks}. Such a block has a very restricted weight
pattern and forces additional weight on both sides.

\begin{lemma}
\label{sixblocks25}
Let $k\ge7$. If $W_i(6)=25$, then
$w_iw_{i+1}w_{i+2}w_{i+3}w_{i+4}w_{i+5}
\in\{444463,364444\}$.
Moreover, $w_{i-1}\ge6$, $w_{i+6}\ge6$ and
$W_{i-6}(6)\ge27$, $W_{i+6}(6)\ge27$.
\end{lemma}

\begin{proof}
Since $W_i(6)=W_i(3)+W_{i+3}(3)=25$, one of the summands equals $12$. By symmetry, suppose that $W_i(3)=12$. Hence
$w_iw_{i+1}w_{i+2}\in\{363,444\}$. The pattern $363$ is impossible as Lemma \ref{tight3windows} guarantees that in this case $W_{i+3}(3)\ge14$. Therefore
$w_iw_{i+1}w_{i+2}=444$.

Since $W_{i+3}(3)=13$ and $W_{i+4}(2)\ge8$, we have
$w_{i+3}\le5$. By Lemma~\ref{weight4}(iv),
$d_{i+1},h_{i+1}\notin D$. If $d_{i+2}\in D$, then
Lemma~\ref{weight4}(iii) gives $w_{i+3}\ge6$, a contradiction. Hence
$d_{i+2},h_{i+2}\notin D$. Since $h_{i+2}$ must be dominated,
$h_{i+3}\in D$, and the $1$-limited constraint gives
$d_{i+3},h_{i+4}\in D$.

Now $w_{i+3}\neq3$ by Lemma~\ref{weight3}. If $w_{i+3}=5$, then
$w_{i+4}=w_{i+5}=4$, contradicting Lemma~\ref{w=5}. Thus
$w_{i+3}=4$. Applying Lemma~\ref{weight4}(ii) to $B_{i+3}$, its alternative
concerning $B_{i+2}$ is impossible since $w_{i+2}=4$. The other
alternative forces five vertices of $B_{i+4}$ into $D$, none of which
is $h_{i+4}$. Since $h_{i+4}\in D$, we obtain $w_{i+4}\ge6$.
As $W_{i+3}(3)=13$, it follows that
$w_{i+4}=6$ and $w_{i+5}=3$. Thus the pattern is $444463$.
The symmetric case gives $364444$.

In either case, Lemmas~\ref{weight3} and~\ref{weight4}(iv) give
$w_{i-1},w_{i+6}\ge6$. Hence, by Lemma~\ref{5blocks},
$W_{i-6}(6)\ge21+6=27$
and 
$W_{i+6}(6)\ge6+21=27$.
\end{proof}

The compensation becomes stronger when two $6$-blocks of weight $25$
occur with one $6$-block between them.

\begin{lemma}
\label{sixblockcompensation}
Let $k\ge18$. If $W_i(6)=W_{i+12}(6)=25$, then
$W_{i+6}(6)\ge28$.
\end{lemma}

\begin{proof}
By Lemma~\ref{sixblocks25}, the equality $W_i(6)=25$ gives
$w_{i+6}\ge6$, while $W_{i+12}(6)=25$ gives
$w_{i+11}\ge6$. Moreover, by Lemma~\ref{8in12}(i),
$W_{i+7}(4)
=
W_{i+7}(2)+W_{i+9}(2)
\ge8+8
=
16$.
Consequently,
$W_{i+6}(6)
=
w_{i+6}+W_{i+7}(4)+w_{i+11}
\ge6+16+6
=
28$.
\end{proof}

The preceding compensation results imply that the average weight of
any sequence of consecutive $6$-blocks is at least $26$.

\begin{lemma}
\label{vsota6}
Let $n\ge2$ and $k\ge6n$. Then
$W_i(6n)\ge26n$
for every $i\in[k]$.
\end{lemma}

\begin{proof}
Write
\[
W_i(6n)=\sum_{j=0}^{n-1}W_{i+6j}(6).
\]
Since $6n\le k$, the $n$ $6$-blocks in this decomposition are pairwise
disjoint. We show that their average weight is at least $26$.

By Lemma~\ref{sixblocks25}, every $6$-block of weight $25$ has each
adjacent $6$-block in the decomposition of weight at least $27$.
Associate each $6$-block of weight $25$ with one such adjacent block.
A block can be associated with at most two blocks of weight $25$. If
it is associated with two, then $n\ge3$ and it lies between them, so
Lemma~\ref{sixblockcompensation} gives weight at least $28$.

Thus a block associated with one block of weight $25$ contributes a
surplus of at least one above $26$, while one associated with two
contributes a surplus of at least two. By Lemma~\ref{6blocks}, the only
possible deficit below $26$ is one, coming from a block of weight $25$.
Hence the total surplus compensates the total deficit, and therefore
$W_i(6n)\ge26n$.
\end{proof}

We next consider the case in which every $2$-block has weight at most
$9$. The local restrictions then force a periodic weight pattern.

\begin{lemma}
\label{total2}
Let $k\ge6$. If $W_i(2)\le9$ for every $i\in[k]$, then
\[
\gammaenaL(G_k)\geq 
\begin{cases}
\frac{13}{3}k; & k\equiv0\pmod3\\
\frac{13}{3}k+\frac23; & k\equiv1\pmod3\\
\frac{13}{3}k+\frac43; & k\equiv2\pmod3.
\end{cases}
\]
\end{lemma}

\begin{proof}
By Lemma~\ref{8in12}, $W_i(2)\in\{8,9\}$ for every $i\in[k]$.
Suppose first that $W_i(2)=9$ for every $i$. Then
$2|D|=\sum_{i=1}^k W_i(2)=9k$,
so $k$ is even and $|D|=\frac92k$. Since $k\ge6$, this is at least the
claimed lower bound in each congruence class modulo $3$.

Suppose now that $W_j(2)=8$ for some $j$. By
Lemma~\ref{8in12}(i), we have $w_jw_{j+1}=44$. Since every
$2$-block has weight at most $9$, we have
$w_{j-1},w_{j+2}\le5$. Neither can be $3$ by Lemma~\ref{weight3}, while either being $4$
would create a pattern $444$, which by Lemma~\ref{weight4}(iv) would
force an adjacent block of weight at least $6$, contradicting
$W_\ell(2)\le9$. Hence $w_{j-1}=w_{j+2}=5$.

Now $w_{j-2},w_{j+3}\le4$. By Lemma~\ref{weight3}, neither can be
$3$. Therefore
\[
w_{j-2}w_{j-1}w_jw_{j+1}w_{j+2}w_{j+3}=454454.
\]

By Lemma~\ref{weight4}(iii), either
$d_j,h_j,d_{j+1},h_{j+1}\in D$,
or none of these four vertices belongs to $D$. In the latter case,
$h_{j+2}\in D$, contradicting Lemma~\ref{w=5}. Hence all four
vertices belong to $D$, and Lemma~\ref{w=5} gives
$h_{j+2}\notin D$.

We claim that $w_{j+4}=4$. Since $w_{j+3}=4$ and
$W_{j+3}(2)\le9$, Lemma~\ref{weight3} gives
$w_{j+4}\in\{4,5\}$. Suppose that $w_{j+4}=5$. Then
$w_{j+5}=4$ by the upper bound on $2$-blocks and
Lemma~\ref{weight3}. Lemma~\ref{w=5} gives
$h_{j+4}\notin D$. Since also $h_{j+2}\notin D$, the $1$-limited
constraint at $h_{j+3}$ implies $h_{j+3}\notin D$. Hence
$d_{j+3}\notin D$ by Lemma~\ref{weight4}(i), leaving
$h_{j+3}$ undominated, a contradiction. Therefore $w_{j+4}=4$, and
so $w_{j+3}w_{j+4}=44$.

Repeating the same argument cyclically yields the periodic pattern
$445445\cdots445$.
Therefore $k$ must be a multiple of $3$. Writing $k=3m$, we obtain
$|D|=13m=\frac{13}{3}k$.
Thus necessarily $k\equiv0\pmod3$, and the claimed lower bound follows.
\end{proof}

We now consider the complementary case in which some $2$-block has
weight at least $10$, while all $5$-blocks remain light.

\begin{lemma}\label{total=5}
Let $k\ge6$ and suppose that $W_t(2)\ge10$ for some $t\in[k]$. If
$W_i(5)\le22$ for every $i\in[k]$, then
\[
\gammaenaL(G_k)\geq
\begin{cases}
\frac{13}{3}k, & k\equiv0\pmod3,\\
\frac{13}{3}k+\frac23, & k\equiv1\pmod3,\\
\frac{13}{3}k+\frac43, & k\equiv2\pmod3.
\end{cases}
\]
\end{lemma}

\begin{proof}
By Lemma~\ref{5blocks}, $W_i(5)\in\{21,22\}$ for every $i\in[k]$.
Let $a\in\{0,1,2\}$ be the residue of $k$ modulo $3$. Suppose, to the
contrary, that
\[
|D|<\tfrac{13}{3}k+\tfrac{2a}{3}.
\]

Since
$W_t(5)=W_t(2)+W_{t+2}(3)\le22$,
Lemma~\ref{8in12} gives
$W_t(2)=10$ and $W_{t+2}(3)=12$. Hence
$w_{t+2}w_{t+3}w_{t+4}\in\{363,444\}$. The pattern $363$ would give
$w_{t+1},w_{t+5}\ge6$ and hence $W_{t+1}(5)\ge24$, a contradiction.
Thus
$w_{t+2}w_{t+3}w_{t+4}=444$.

By Lemma~\ref{weight4}(iv), one of $w_{t+1},w_{t+5}$ is at least $6$.
Reversing the cyclic order if necessary, assume that $w_{t+5}\ge6$.
Since $W_{t+1}(5)\le22$, we have $w_{t+1}\le4$. Moreover,
$w_{t+1}\neq3$ by Lemma~\ref{weight3}, so $w_{t+1}=4$. Consequently,
$w_t=6$, and again from $W_{t+1}(5)\le22$ we obtain $w_{t+5}=6$.
Thus
$w_tw_{t+1}w_{t+2}w_{t+3}w_{t+4}w_{t+5}=644446$.

Now $W_{t+3}(5)\le22$ and $W_{t+6}(2)\ge8$ give
$w_{t+6}=w_{t+7}=4$. Similarly, using Lemma~\ref{weight3} to exclude
weight $3$, the bounds $W_{t+4}(5),W_{t+5}(5)\le22$ give
$w_{t+8}=w_{t+9}=4$. Applying Lemma~\ref{weight4}(iv) to the pattern
$w_{t+7}w_{t+8}w_{t+9}=444$ gives $w_{t+10}\ge6$, while
$W_{t+6}(5)\le22$ gives $w_{t+10}\le6$. Hence $w_{t+10}=6$.

Repeating this argument cyclically yields the periodic pattern
$6444464444\cdots64444$. Therefore $k$
must be a multiple of $5$. Each period has total weight $22$, and hence
$|D|=\frac{22}{5}k$.
The assumed inequality therefore gives
\[
\tfrac{22}{5}k<\tfrac{13}{3}k+\tfrac{2a}{3},
\]
or equivalently $k<10a$.

If $a=0$, this is impossible. If $a=1$, then $k<10$, contradicting
$5\mid k$ and $k\ge6$. Finally, if $a=2$, then $k<20$ and $5\mid k$,
so $k\in\{10,15\}$, but neither is congruent to $2$ modulo $3$.
Thus the assumed strict inequality is impossible, and the claimed
lower bound follows. 
\end{proof}

It remains to consider the case in which some $5$-block has weight at
least $23$. The following estimates will handle the remaining small
values of $k$.

\begin{lemma}
\label{obstaja23}
Let $k\ge6$ and suppose that $W_t(5)\ge23$ for some $t\in[k]$.
\begin{itemize}
    \item[(i)] If $k\ge7$, then $W_i(7)\ge30$ for every $i\in[k]$,
    and $W_j(7)\ge31$ for some $j\in[k]$.
    \item[(ii)] If $k\ge8$, then $W_j(8)\ge36$ for some $j\in[k]$.
    \item[(iii)] If $k\ge9$, then $W_i(9)\ge38$ for every $i\in[k]$,
    and $W_j(9)\ge39$ for some $j\in[k]$.
    \item[(iv)] If $k\ge10$, then $W_j(10)\ge44$ for some $j\in[k]$.
    \item[(v)] If $k\ge11$, then $W_j(11)\ge49$ for some $j\in[k]$.
\end{itemize}
\end{lemma}

\begin{proof}
(i) Suppose, to the contrary, that
$W_i(7)=w_i+W_{i+1}(6)\le29$ for some $i\in[k]$. If
$W_{i+1}(6)=25$, then Lemma~\ref{sixblocks25} gives $w_i\ge6$, and
hence $W_i(7)\ge31$, a contradiction. Thus $W_{i+1}(6)\ge26$, so
$w_i=3$ and $W_{i+1}(6)=26$. Lemma~\ref{weight3} gives
$w_{i+1}\ge6$, and hence $W_{i+2}(5)\le20$, contradicting
Lemma~\ref{5blocks}. Therefore $W_i(7)\ge30$ for every $i\in[k]$.
Moreover, $W_t(7)=W_t(5)+W_{t+5}(2)\ge23+8=31$.

\medskip
(ii) Suppose, to the contrary, that $W_i(8)\le35$ for every
$i\in[k]$. By (i), $w_i=W_i(8)-W_{i+1}(7)\le5$ for every $i$. Since a block of
weight $3$ would, by Lemma~\ref{weight3}, force an adjacent block to
have weight at least $6$, we have $w_i\in\{4,5\}$ for every $i$.
Furthermore, $W_{t+5}(3)=W_t(8)-W_t(5)\le12$, so
Lemma~\ref{8in12}(ii) gives $w_{t+5}w_{t+6}w_{t+7}=444$.
Lemma~\ref{weight4}(iv) then forces an adjacent block to have weight
at least $6$, a contradiction.

\medskip
(iii) By (i) and Lemma~\ref{8in12}(i),
$W_i(9)=W_i(7)+W_{i+7}(2)\ge38$ for every $i\in[k]$. If
$W_i(9)=38$ for every $i$, then $W_{i+7}(2)=8$ for every $i$.
Lemma~\ref{8in12}(i) then gives $w_iw_{i+1}=44$ for every $i$, and
hence $W_t(5)=20$, contradicting $W_t(5)\ge23$.

\medskip
(iv) By Lemma~\ref{5blocks},
$W_t(10)=W_t(5)+W_{t+5}(5)\ge23+21=44$.

\medskip
(v) Suppose, to the contrary, that $W_i(11)\le48$ for every
$i\in[k]$. Then
$W_{t-6}(6)+W_t(5)\le48$ and $W_t(5)+W_{t+5}(6)\le48$.
Since every $6$-block has weight at least $25$ and $W_t(5)\ge23$,
equalities hold throughout. Thus
$W_{t-6}(6)=25$, $W_t(5)=23$, and $W_{t+5}(6)=25$.
Lemma~\ref{sixblocks25} gives $w_t,w_{t+4}\ge6$, and therefore
$W_{t+1}(3)=W_t(5)-w_t-w_{t+4}\le11$, contradicting
Lemma~\ref{8in12}(ii).
\end{proof}

We are now ready to combine the preceding local and global estimates
to determine the exact value of $\gammaenaL(G_k)$.

\begin{theorem}
\label{glavnithm}
For the Goldberg snark $G_k$, $k\ge3$, we have
$$\gammaenaL(G_k)=
\begin{cases}
\frac{13}{3}k, & k\equiv0\pmod{3},\\
\frac{13}{3}k+\frac{2}{3}, & k\equiv1\pmod{3},\\
\frac{13}{3}k+\frac{4}{3},
    & k\equiv2\pmod{3}\text{ and }k\neq5,\\
22, & k=5.
\end{cases}$$
\end{theorem}

\begin{proof}
Let $D$ be a minimum $1$-limited dominating set of $G_k$.

Suppose first that $k=3$. If $w_i=3$ for some $i$, then
Lemma~\ref{weight3} gives
$|D|=w_{i-1}+w_i+w_{i+1}\ge6+3+6>13$.
Otherwise $w_i\ge4$ for every $i$. If $|D|=12$, then the weight
pattern is $444$, contradicting Lemma~\ref{weight4}(iv). Hence
$|D|\ge13$.
For $k=4$, Lemma~\ref{4blocks} gives
$|D|=W_j(4)\ge18$, while for $k=5$, Lemma~\ref{5blocks} gives
$|D|=W_j(5)\ge22$. 

Assume now that $k\ge6$. If $W_i(2)\le9$ for every $i\in[k]$, the
required lower bound follows from Lemma~\ref{total2}. Hence we may
assume that $W_j(2)\ge10$ for some $j\in[k]$. If $W_i(5)\le22$ for
every $i\in[k]$, the result follows from Lemma~\ref{total=5}. Thus we
may further assume that $W_t(5)\ge23$ for some $t\in[k]$.

For $k=6$, Lemma~\ref{6blocks} gives an index $\ell$ such that
$W_{\ell}(6)\ge26$. Since a $6$-block contains all blocks of $G_6$,
$|D|=W_{\ell}(6)\ge26$. 
For $7\le k\le11$, the required lower bounds follow from
the corresponding parts of Lemma~\ref{obstaja23}, while for $k=12$,
Lemma~\ref{vsota6} gives $|D|=W_i(12)\ge52$. 


It remains to consider $k\ge13$. Write
$k=6n+a$, where $a\in[6]$. Then $n\ge2$.
To apply Lemma~\ref{vsota6} to the remaining $6n$ blocks, it suffices
to find $a$ consecutive blocks with sufficiently large total weight.
We claim that, for each $a\in[6]$, there exists $s\in[k]$ such that
$W_s(a)\ge c_a$, where
\[
\begin{array}{c|cccccc}
a&1&2&3&4&5&6\\ \hline
c_a&5&10&13&18&23&26.
\end{array}
\]
Indeed, $W_j(2)\ge10$ implies that at least one of the blocks $B_j$ and $B_{j+1}$ has weight at
least $5$, giving the case $a=1$, while the case $a=2$ is immediate.
By Lemmas~\ref{wi3} and~\ref{8in12},
$W_j(3)=W_j(2)+w_{j+2}\ge10+3=13$
and
$W_j(4)=W_j(2)+W_{j+2}(2)\ge10+8=18$.
The case $a=5$ follows from $W_t(5)\ge23$, and finally
$W_t(6)=W_t(5)+w_{t+5}\ge23+3=26$.

The chosen $a$ consecutive blocks and the remaining $6n$ consecutive
blocks partition the cyclic sequence of all blocks. Hence, by
Lemma~\ref{vsota6},
$|D|=W_s(a)+W_{s+a}(6n)\ge c_a+26n$.
Since
\[
c_a=
\begin{cases}
\frac{13a}{3}, & a\equiv0\pmod3,\\
\frac{13a+2}{3}, & a\equiv1\pmod3,\\
\frac{13a+4}{3}, & a\equiv2\pmod3,
\end{cases}
\]
and $k=6n+a$, this is precisely the claimed lower bound for
$\gammaenaL(G_k)$. The corresponding upper bound follows from
Lemma~\ref{upper}, and hence equality holds in every case.
\end{proof}

\section{Conclusion}

In this paper, we proved that \textsc{1-Limited Dominating Set} is
$\mathsf{NP}$-complete even on $2$-connected planar cubic graphs.
Since $1$-limited domination and $(1,2)$-domination coincide on cubic
graphs, the same holds for \textsc{$(1,2)$-Dominating Set}. Together with the known results for fixed $k\ge2$ from \cite{San-samost}, this completes the complexity classification of \textsc{$k$-Limited Dominating Set} for every fixed positive integer $k$.

The hardness result further motivates the study of exact values on
well-structured graph families, which may also serve as test cases for
broader conjectures on cubic graphs. In this direction, we determined
the exact $1$-limited domination number, equivalently the exact
$(1,2)$-domination number, of the Goldberg family $G_k$ for every
$k\ge3$.

Our exact result provides further evidence for several such open
problems. In \cite{TepehErves}, it was conjectured that every connected
cubic graph $G$ of order $n$, except for a specific exceptional graph
of that order (called the trilobite), satisfies
$\gamma_{1,2}(G)\le \cind(G)$. For the Goldberg family, the vertices
$\{a_i,b_i,f_i,g_i,h_i:i\in[k]\}$ induce
$C_{2k}\cup C_{2k}\cup C_k$, and hence
$\cind(G_k)\ge5k=\frac58|V(G_k)|$. Together with our exact value of
$\gamma_{1,2}(G_k)$, this yields
$\gamma_{1,2}(G_k)\le\cind(G_k)$. In particular, the Goldberg family
satisfies the conjecture of Henning et al.~\cite{HenningJoos} that
$\cind(G)\ge |V(G)|/2$ for every cubic graph $G$.

Finally, Fakhran et al.~\cite{fakhran} asked whether
$\gamma_{1,2}(G)\le\frac23|V(G)|$ holds for every triangle-free cubic
graph $G$. Since $G_k$ is triangle-free for $k\ge4$, our exact formula
gives $\gamma_{1,2}(G_k)<\frac23|V(G_k)|$ for every such $k$.
Thus the Goldberg family also provides further evidence for this
proposed upper bound.

\vskip 1pc \noindent{\bf Acknowledgments.} 
This work has been supported by the European Commission's Horizon Europe Research and Innovation programme through the Marie Sk\l{}odowska-Curie Actions Staff Exchanges (MSCA-SE) under Grant Agreement no.~101182819 (COVER: (C)ombinatorial (O)ptimization for (V)ersatile Applications to (E)merging u(R)ban Problems).
A.~T.~was partially supported by the Slovenian Research and Innovation Agency (ARIS), program no.~P1--0297 and project no.~J1--70016. 

\vskip 1pc \noindent{\bf Statements and Declarations.}
The author has no conflict of interest to declare. 

\vskip 1pc \noindent{\bf Declaration of generative AI and AI-assisted technologies
in the manuscript preparation process}.
During the preparation of this work, the authors used ChatGPT (OpenAI)
to assist with language editing. After using this tool, the authors
reviewed and edited the content as needed and take full responsibility
for the content of this paper.

\end{document}